\documentclass[12pt]{amsart}
\usepackage{amssymb}
\usepackage{hyperref}
\usepackage{tabularx}
\usepackage{booktabs}
\usepackage{caption}

\newtheorem{thm}{Theorem}[section]
\newtheorem{lem}[thm]{Lemma}
\newtheorem{cor}[thm]{Corollary}
\newtheorem{prop}[thm]{Proposition}
\newtheorem{ex}[thm]{Example}

\newtheorem*{prob*}{Open problem}

\theoremstyle{definition}

\newtheorem{defi}[thm]{Definition}

\theoremstyle{remark}

\newtheorem{rem}[thm]{Remark}
\newtheorem*{rem*}{Remark}

\DeclareMathOperator{\id}{id}

\newcommand{\kringel}{\mathbin{\raise1pt\hbox{$\scriptstyle\circ$}}}
\newcommand{\pkt}{\mathbin{\raise0pt\hbox{$\scriptstyle\bullet$}}}

\newcommand{\map}[3]{ #1 : #2 \longrightarrow #3 }

\newcommand{\C}{\mathbb{C}}

\newcommand{\F}{\mathbb{F}}
\newcommand{\N}{\mathbb{N}}

\newcommand{\Q}{\mathbb{Q}}

\newcommand{\Z}{\mathbb{Z}}

\newcommand{\per}{{\rm per}}

\newcommand{\Der}{{\rm Der}}

\newcommand{\diag}{\mathop{\rm diag}}

\newcommand{\Lg}{\mathfrak{g}}
\newcommand{\Lh}{\mathfrak{h}}

\newcommand{\CB}{\mathcal{B}}

\newcommand{\CN}{\mathcal{N}}

\newcommand{\CP}{\mathcal{P}}

\newcommand{\abs}[1]{\lvert#1\rvert}

\newcommand{\al}{\alpha}
\newcommand{\be}{\beta}
\newcommand{\ga}{\gamma}
\newcommand{\de}{\delta}

\newcommand{\la}{\lambda}
\newcommand{\om}{\omega}
\newcommand{\si}{\sigma}

\newcommand{\ov}{\overline}

\newcommand{\ra}{\rightarrow}

\renewcommand{\phi}{\varphi}

\begin{document}


\title[Polynomial derivation identity]{The structure of Lie algebras with a derivation satisfying a polynomial identity}

\author[D. Burde]{Dietrich Burde}
\author[W. Moens]{Wolfgang Alexander Moens}
\address{Fakult\"at f\"ur Mathematik\\
Universit\"at Wien\\
  Oskar-Morgenstern-Platz 1\\
  1090 Wien \\
  Austria}
\email{dietrich.burde@univie.ac.at}
\address{Fakult\"at f\"ur Mathematik\\
Universit\"at Wien\\
  Oskar-Morgenstern-Platz 1\\
  1090 Wien \\
  Austria}
\email{wolfgang.moens@univie.ac.at}

\date{\today}


\begin{abstract}
We prove nilpotency results for Lie algebras over an arbitrary field admitting a derivation, which satisfies a given
polynomial identity $r(t)=0$. In the special case of the polynomial $r=t^n-1$ we obtain
a uniform bound on the nilpotency class of Lie algebras admitting a periodic derivation of order $n$.
We even find an optimal bound on the nilpotency class in characteristic $p$ if $p$ does not divide
a certain invariant $\rho_n$. 
We are able to generalize the results to Lie rings over $\Z$.
 \end{abstract}

\maketitle

\section{Introduction}

Let $\Lg$ be a finite-dimensional Lie algebra over an arbitrary field $K$. 
A derivation $D$ of $\Lg$ is said to be {\em nonsingular} if it is bijective as a linear transformation.
It is called {\em periodic of order $n$}, if $D^n=\id$ and $D^k\neq \id$ for all $0<k<n$.
The interest in nonsingular or periodic derivations comes in part from the coclass theory for groups and Lie algebras,
with the proof of the coclass conjectures by Shalev. 
In many cases the existence of a nonsingular or periodic derivation has a strong impact on the structure of the
Lie algebra. For example, by a result of Jacobson \cite{JAC},
a Lie algebra over a field of characteristic zero admitting a nonsingular derivation is {\em nilpotent}. 
Furthermore, by a result of Kostrikin and Kuznetsov \cite{KUK}, a Lie algebra over a field of characteristic zero
admitting a periodic derivation of order $n$ such that $6$ does not divide $n$ is {\em abelian}.
In prime characteristic $p>0$ however, the situation is more complicated. There exist even simple modular 
Lie algebras admitting a periodic derivation. Shalev asked in his Problem $1$ in \cite{SHA}, which positive integers
$n$ arise as the order of a periodic derivation for a non-nilpotent Lie algebra in characteristic $p$.
The set of such integers was denoted $\CN_p$ by Mattarei, who showed in \cite{MA3} that $\CN_p$ coincides with
the set of all positive integers $n$ such that there exists an element $\al\in \ov{\F}_p$ with
$(\al+\la)^n=1$ for all $\la\in \F_p$. In  \cite{MA1,MA2,MA3}  Mattarei has made a detailed study of the set
$\CN_p$, using interesting number theoretical methods. We refer to the introduction of \cite{MA2} for a nice
overview. \\[0.2cm]
In this article we study not only Lie algebras admitting periodic derivations, but more generally Lie algebras over a
field $K$ admitting a derivation $D$, which satisfies an {\em arbitrary} polynomial identity $r(D)=0$ given by a polynomial
$r\in K[t]$. We obtain a general result on the nilpotency
by applying recent work of \cite{MON,MOE} as follows, see Theorem $\ref{3.10}$ and Remark $\ref{3.11}$. \\[0.3cm]
{\bf Theorem} {\em
Let $K$ be a field of characteristic $p\ge 0$, $r\in K[t]$ be a polynomial of degree $n\ge 0$ and
$X=\{\al \in \ov{K} \mid r(a)=0 \}$ be the set of roots in $\ov{K}$.
If $X$ is an arithmetically-free subset of $(\ov{K},+)$, then every Lie algebra $\Lg$ over $K$ admitting
a derivation $D$, which satisfies $r(D)=0$, is nilpotent of class $c(\Lg)\le H(n)$.
If $X$ is not an arithmetically-free subset of $(\ov{K},+)$, then there exists
some non-nilpotent Lie algebra over $\ov{K}$ of dimension $n+1$ admitting a derivation $D$, which satisfies $r(D)=0$.}\\[0.3cm]
Here $H(n)$ is the generalized Higman map, see section $3$. In characteristic zero the root set $X$ is arithmetically-free
if and only if $r(0)\neq 0$. In characteristic $p>0$ this is much harder to describe. But if we specify to the
case $r=t^n-1\in \F_p[t]$, i.e., to periodic derivations of order dividing $n$, we can show that
$X=X_{n,p}=\{\al\in \ov{\F}_p \mid \al^n=1 \}$ is not an arithmetically-free subset of $(\ov{\F}_p,+)$
if and only if $n\in \CB_p$, where
\[
\CB_p =\N\cdot \{\per(h (t^p-t))\mid h\in \F_p[t] \text{ with } h(0)\neq 0, \deg(h)\ge 1\}.
\]
Here $\per(r)$ denotes the {\em period} of $r\in \F_p[t]$, which is the minimal positive integer $m$ such that
 $r$ divides $t^m-1$ in $\F_p[t]$, if such an $m$ exists.  
In fact, we prove the following result, see Proposition $\ref{3.18}$. \\[0.3cm]
{\bf Proposition} {\em
Let $n\in \N$ and $p$ be a prime number. If $n\not\in \CB_p$, then every Lie algebra $\Lg$ over
a field of characteristic $p>0$ admitting a periodic derivation of order $n$ is nilpotent of class
$c(\Lg)\le H(n)$. If  $n\in \CB_p$ then there exists some non-nilpotent Lie algebra in characteristic
$p>0$ admitting a periodic derivation of order $n$.}\\[0.3cm]  
In paricular we see that the set $\CN_p$ considered by Shalev and Mattarei coincides with our set $\CB_p$.
Any $h\in \F_p[t]$ with $h(0)\neq 0$ and $\deg(h)\ge 1$ will produce an element of $\CB_p=\CN_p$ by computing the
period of $h(t^p-t)$. For example, looking at irreducible polynomials $h\in \F_2[t]$ of low degree we see that
$3,7,31,73,85,127$ are the first few primitive elements of $\CN_2$, see Example $\ref{3.16}$.
However, most importantly we obtain a uniform upper bound $H(n)$ for the nilpotency class
in case that $n\not\in \CB_p$. In addition, we can improve this bound significantly by excluding finitely many prime
characteristics $p>0$. In Theorem $\ref{3.6}$ we show the following. \\[0.3cm]
{\bf Theorem } {\em
Let $\Lg$ be a Lie algebra  over a field $K$ of arbitrary characteristic $p\ge 0$. Suppose that $\Lg$ admits a periodic
derivation $D$ of order $n$ such that $p$ does not divide $\rho_n$. Then $\Lg$ is nilpotent of class $c(\Lg)\le 1$ if
$n$ is not divisible by $6$ and of class $c(\Lg)\le 2$ if $6\mid n$.} \\[0.3cm]
The invariant $\rho_n$ is defined by the resultant of the polynomials $t^n-1$ and $(t+1)^n-1$ if $n$ is not divisible by $6$,
and by the resultant of $\frac{t^n-1}{\Phi_3}$ and $\frac{(t+1)^n-1}{\Phi_3}$ if $6\mid n$, where $\Phi_3$ is the third
cyclotomic polynomial in $\Z[t]$. It is known that $\rho_n$ in the first case is given by the Wendt determinant of a
circulant matrix with first row the binomial coefficients. This was first studied by Wendt in \cite{WEN} in connection with
Fermat's last theorem. For our result, the prime divisors of $\rho_n$ are of interest. We show some properties of $\rho_n$
in section $2$. In case of $p\mid \rho_n$ the conclusion of Theorem  $\ref{3.6}$ may or may not hold. Moreover for many
primes $p$ dividing $\rho_n$ the set $X_{n,p}$ is in fact still arithmetically-free, so that $\Lg$ is nilpotent. \\[0.2cm]
Finally, we can generalize our result to derivations of Lie rings over $\Z$ satisfying an arbitrary polynomial identity.
Here we need new integer valued invariants $\delta(r)$ and $\sigma(r)$ for $r\in \Z[t]$, which we study in section $2$.
For the specific example of $r=t^n-1$ these invariants are given by the discriminant of $r$ for $\delta(r)$, and by
$\sigma(r)=(-\rho_n)^n$ if $n$ is not divisible by $6$, and by $\sigma(r)=\left( \frac{n^2\rho_n}{3}\right)^n$ if $6\mid n$.

\section{Arithmetic Invariants}

Let $S$ be a commutative ring and $f,g\in S[t]$ be two polynomials. Denote by $R(f,g)$ the {\em resultant} of $f$ and $g$
over $S$, given by the determinant of the Sylvester matrix with columns given by the coefficients of $f$ and $g$.
Recall that the resultant of two polynomials with coefficients in an integral domain is zero if and only if they have a common
divisor of positive degree. 

\begin{lem}
Let $n\ge 1$. The greatest common divisor of  $f=t^n-1, g=(t+1)^n-1 \in \Z[t]$ is given by
\[
\gcd(f,g)=\begin{cases} 1 & \text{ if } n\not\equiv 0 \bmod 6,\\
 t^2+t+1 & \text{ if } n\equiv 0\bmod 6.
 \end{cases}
\]
Here $\Phi_3=t^2+t+1\in \Z[t]$ is the third cyclotomic polynomial.
\end{lem}

\begin{proof}
Let $h=\gcd (f,g)$. Since the discriminant of $t^n-1$ is nonzero in characteristic zero, every irreducible factor of $h$ has
multiplicity one. For the first case let $n\not\equiv 0 \bmod 6$ and assume that $h\neq 1$. Then $f$ and $g$ have a
common root $\ga$, so that $\ga^n=(\ga+1)^n=1$. By Lemma $2.2$ of \cite{BU45} with $\alpha=1$ and $\be=\ga$
it follows that $\ga=\om$ is a primitive third root of unity, and hence $\ga+1$ is a primitive sixth root of unity.
Because of $(\ga+1)^n=1$ we obtain $6\mid n$, which is a contradiction. Hence $h=1$ in this case.\\
In the second case we assume that $6\mid n$. Let $\om$ be a primitive third root of unity. Then $\om$ is a root of $t^n-1$
and of $(t+1)^n-1$, since $1+\om$ is a primitive sixth root of unity. So we have $\Phi_3\mid h$. By Lemma $2.2$ of \cite{BU45}
every root of $h$ is also a root of $\Phi_3$. Since $h$ has only simple roots it follows that $h=\Phi_3$.
\end{proof}

The following definition yields a nonzero integer, because the  polynomials are coprime in each case by the above lemma.

\begin{defi}
For $n\ge 1$ define a nonzero integer $\rho_n$ by
\[
\rho_n=  \begin{cases} R(t^n-1,(t+1)^n-1) &  \text{ if } n\not\equiv 0 \bmod 6,\\[0.2cm]
R\left(\frac{t^n-1}{\Phi_3}, \frac{(t+1)^n-1}{\Phi_3}\right) &  \text{ if } n\equiv 0\bmod 6.
 \end{cases}
\]  
\end{defi}  

The resultant of $t^n-1$ and $(t+1)^n-1$ has been studied for a long time already. It arises among other things
in number theory.

\begin{prop}
For $n\not\equiv 0\bmod 6$ the invariant $\rho_n$ is given by the Wendt determinant $\det (C(n))$, where $C(n)\in M_n(\Z)$ is
the following circulant matrix: 
\[
  C(n)= \begin{pmatrix} 1 & \binom{n}{1} & \binom{n}{2} & \cdots &  \binom{n}{n-1} \\[0.2cm]
    \binom{n}{n-1} & 1 & \binom{n}{1} & \cdots & \binom{n}{n-2} \\[0.2cm]
    \cdots &  \cdots &  \cdots &  \cdots &  \cdots \\[0.2cm]
    \binom{n}{1} & \binom{n}{2} & \binom{n}{3} & \cdots  & 1
    \end{pmatrix}
 \]
We have $\det(C(n))=0$ if and only if $n\equiv 0\bmod 6$. 
\end{prop}  

\begin{proof}
The circulant matrix $C(n)$ with with first row the binomial coefficients was introduced by E. Wendt in \cite{WEN}
in connection with Fermat's last theorem. Wendt showed that its determinant equals the resultant of the
polynomials $t^n-1$ and $(t+1)^n-1$. E. Lehmer proved in \cite{LEH} that $\det(C(n))=0$ if and only if $n\equiv 0\bmod 6$.  
\end{proof}

For $n\not \equiv 0\bmod 6$ the above $\rho_n$ coincides with the invariant $\Delta_n$ introduced by
Kostrikin and Kuznetsov in \cite{KUK}. The Wendt determinant is listed at OEIS as sequence $A048954$. The first ten numbers
and their prime factorization are given as follows.

\vspace*{0.5cm}
\begin{center}
\begin{tabular}{c|c|c}
$n$ & $\rho_n$ & prime factors \\
\hline
$1$ & $1$ & $1$ \\
$2$ & $-3$ & $-3$ \\
$3$ & $28$ & $2^{2}\cdot 7$ \\
$4$ & $-375$ & $-3\cdot 5^{3}$  \\
$5$ & $3751$ & $11^2\cdot 31$ \\
$7$ & $6835648$ & $2^{6}\cdot 29^{2}\cdot 127$ \\
$8$ & $ -1343091375$ & $-3^7\cdot 5^3\cdot 17^3$ \\
$9$ & $364668913756$ & $2^2\cdot 7\cdot 19^4\cdot 37^2\cdot 73$ \\
$10$ & $-210736858987743$ & $-3\cdot 11^9\cdot 31^3$ \\
$11$ & $101832157445630503$ & $23^5\cdot 67^2\cdot 89\cdot 199^2$ 
\end{tabular}
\end{center}
\vspace*{0.5cm}

For the case $n\equiv 0\bmod 6$ there is no sequence available. We list a few numbers with their
prime decomposition.

\vspace*{0.5cm}
\begin{center}
\begin{tabular}{c|c}
$n$ & $\rho_n$ \\
\hline
$6$ & $-2^2\cdot 3\cdot 7^3$ \\
$12$ & $-2^{10}\cdot 3\cdot 5^3\cdot 7^3\cdot 13^9$ \\
$18$ & $-2^{2}\cdot 3^{12}\cdot 7^3\cdot 19^{15}\cdot 37^6\cdot 73^3$ \\
$24$ & $-2^{30}\cdot 3^{31}\cdot 5^{21}\cdot 7^9\cdot 13^9\cdot 17^3\cdot 73^6\cdot 241^3$  \\
$30$ & $-2^{50}\cdot 3^{1}\cdot 5^{8}\cdot 7^3\cdot 11^9\cdot 31^{27}\cdot 61^{12}\cdot 151^3\cdot 271^6\cdot 331^3$ \\
$36$ & $-2^{10}\cdot 3^{12}\cdot 5^{3}\cdot 7^3\cdot 13^9\cdot 17^6\cdot 19^{15}\cdot 37^{33}\cdot 73^{15}\cdot 109^9
       \cdot 181^6\cdot 757^6$ \\
\end{tabular}
\end{center}
\vspace*{0.5cm}
The integers $\rho_n$ satisfy the following divisibility property.

\begin{lem}\label{2.4}
Let $m,n\ge 1$ with $m\mid n$. Then $\rho_m\mid \rho_n$ in $\Z$.
\end{lem}

\begin{proof}
The proof is a case-by-case verification. Suppose first that $6$ does not divide $n$. Then $6$ also does not divide
$m$. We have $(t^m-1)\mid (t^n-1)$ in $\Z[t]$ because of $m\mid n$ in $\Z$. Then we obtain $((t+1)^m-1)\mid ((t+1)^n-1)$, 
and by the multiplicative property of resultants we have
\[
\rho_m=R(t^m-1,(t+1)^m-1) \mid R(t^n-1,(t+1)^n-1)=\rho_n.
\]
Now suppose that $6\mid n$ and $6\mid m$. Then
$
\frac{t^m-1}{\Phi_3} \mid \frac{t^n-1}{\Phi_3}$ and $\frac{(t+1)^m-1}{\Phi_3} \mid \frac{(t+1)^n-1}{\Phi_3}$
in $\Z[t]$, so that again $\rho_m\mid \rho_n$. For the last case assume that  $6\mid n$ but $6$ does not divide $m$.
Then $\Phi_3$ does not divide $t^m-1$ and not $(t+1)^m-1$, but it does divide $t^n-1$ and $(t+1)^n-1$. So $t^m-1$ divides
$\frac{t^n-1}{\Phi_3}$ and $(t+1)^m-1$ divides $\frac{(t+1)^n-1}{\Phi_3}$, so that $\rho_m\mid \rho_n$.
\end{proof}

We also need the following lemma.

\begin{lem}\label{2.5}
Let $K$ be a field of arbitrary characteristic $p\ge 0$. Suppose that $\al\in K$ is a common root of $t^n-1$ and
$(t+1)^n-1$. If $6$ does not divide $n$ then $p\mid \rho_n$. If $6\mid n$ then $p\mid \rho_n$ or $\Phi_3(\al)=0$.
\end{lem}  

\begin{proof}
Assume that $6$ does not divide $n$. There exist $u,v\in \Z[t]$ such that
\[
\rho_n=u\cdot (t^n-1)+v\cdot ((t+1)^n-1).
\]
By evaluating in $\al$ we obtain $\rho_n\al=0$ in $K$ and hence $p\mid \rho_n$. Suppose now that $6\mid n$.
Then there exist  $u,v\in \Z[t]$ such that $\rho_n\cdot \Phi_3=u\cdot (t^n-1)+v\cdot ((t+1)^n-1)$
and evaluating in $\al$ yields $\rho_n\cdot \Phi_3(\al)=0$ in $K$ and therefore $p\mid \rho_n$ or $\phi_3(\al)=0$.
\end{proof}

The following integer valued invariant for polynomials has been defined in \cite{MOE}, Definition $1.3.3$.

\begin{defi}
Let $r\in \Z[t]$ be a nonzero polynomial of degree $d$ and leading coefficient $a\in \Z$. Define an invariant
$\delta(r)$ by $\delta(r)=r$ for $d=0$ and by
\[
\delta(r)=a^{1+2d^2}\cdot (m-1)!\cdot \prod_{\substack{1\le i,j\le \ell \\ i\neq j}} (\la_i-\la_j)^m
\]
for $d\ge 1$, where $\la_1,\cdots,\la_{\ell}$ are the distinct roots of $r$ in $\ov{\Q}$ with corresponding multiplicities
$m_1,\ldots ,m_{\ell}$ and $m:=\max\{ m_1,\ldots ,m_{\ell}\}$.
\end{defi}

It is obvious from the definition that $\delta(r)$ is nonzero. Moreover it was shown in \cite{MOE}, Lemma $4.3.1$ that
$\de(r)$ is an {\em integer} since it can be expressed as a determinant of a certain Sylvester matrix of polynomials with
integer coefficients.  The notation in \cite{MOE} for $\delta(r)$ is also ${\rm Discr}_*(r(t))$, indicating that the
invariant is a certain modification of the discriminant of $r$. It can be computed without having to extract roots of the
polynomial $r$. 

\begin{ex}
Let $r=t^n-1$. Then $\delta(r)$ coincides with the usual discriminant of $r$. We have
\begin{align*}
\delta(r) & = \operatorname{disc}(t^n-1) \\
          & = \prod_{1\le i<j\le n}(\zeta^i-\zeta^j)^2 \\
          & = (-1)^{\frac{n(n-1)}{2}}\cdot (-1)^{n-1}\cdot n^n,
\end{align*}
where $\zeta$ is a primitive $n$-th root of unity.
\end{ex}
  
Let us define another integer valued invariant for polynomials here. It is a natural analogue of
the invariant ${\rm Prod}(r(t))$ defined in \cite{MOE}, Definition $1.3.3$.

\begin{defi}
Let $r\in \Z[t]$ be a nonzero polynomial of degree $d$ and leading coefficient $a\in \Z$. Define an invariant $\si(r)$ 
by $\si(r)=1$ for $d=0$ and by 
\begin{align*}
  \si(r) & = a^{2d^3}\cdot \prod_{\substack{1\le i,j\le \ell\\ r(\la_i+\la_j)\neq 0}}r(\la_i+\la_j) \\
         & = a^{2d^3}\cdot \prod_{\substack{1\le i,j,k\le \ell\\ r(\la_i+\la_j)\neq 0}}a\cdot (\la_i+\la_j-\la_k)^{m_k}
\end{align*}
for $d\ge 1$, where $\la_1,\cdots,\la_{\ell}$ are the distinct roots of $r$ in $\ov{\Q}$ with corresponding multiplicities
$m_1,\ldots ,m_{\ell}$.
\end{defi}

Again it is clear from the definition that $\si(r)$ is nonzero. Since we can rewrite  $\si(r)$ as a certain resultant
of two polynomials with integer coefficients (compare with Lemma $4.3.3$
of \cite{MOE}) we see that $\si(r)$ is always an integer. 

\begin{prop}
Let $r=t^n-1$. Then we have
\[
\si (r)=
\begin{cases}
\displaystyle (-\rho_n)^n & \text{ if } n\not\equiv 0\bmod 6,\\[0.3cm]  
\displaystyle \left(\frac{n^2\rho_n}{3}\right)^n &  \text{ if } n\equiv 0\bmod 6.
\end{cases}
\]
\end{prop}

\begin{proof}
Suppose first that $6$ does not divide $n$. Then the sum of two $n$-th roots is never an $n$-th root.
Thus we have  
\begin{align*}
\sigma(r) & = \prod_{0\le i,j\le n-1}((\zeta^i+\zeta^j)^n-1) \\
          & = \prod_{0\le i,j\le n-1}((\zeta^{i-j}+1)^n-1) \\
          & = \prod_{0\le k\le n-1}((\zeta^k+1)^n-1)^n \\
          & = (R((t+1)^n-1,t^n-1))^n \\
          & = (-\rho_n)^n.
\end{align*}
If $6\mid n$ then the sum of two $n$-th roots is an $n$-th root if and only if the ratio is a primitive third
root of unity. So we have
\begin{align*}
\si (r) & = \prod_{\substack{0\le k\le n-1 \\ (\zeta^k+1)^n\neq 1}}((\zeta^k+1)^n-1)^n \\
        & = \prod_{\substack{0\le k\le n-1 \\ \Phi_3(\zeta^k)\neq 0}}((\zeta^k+1)^n-1)^n \\
        & = \left( R\left((t+1)^n-1,\frac{t^n-1}{\Phi_3}\right)\right)^n \\
        & = (-\rho_n)^n\cdot \left( R\left(\Phi_3,\frac{t^n-1}{\Phi_3}\right) \right)^n\\
        & = (-\rho_n)^n\cdot \prod_{\substack{d\mid n\\ d\neq 3}}R(\Phi_3,\Phi_d)^n.
\end{align*}
Now in \cite{APO} it is shown that for integers $n\ge m\ge 1$ we have
\[
R(\Phi_m,\Phi_n)=
\begin{cases} p^{\phi(m)} & \text{ if } \frac{n}{m} \text{ is a power of a prime } p,\\
              1          & \text{ otherwise.}
\end{cases}  
\]  
It is easy to see that this implies that
\[
\prod_{\substack{d\mid n\\ d\neq 3}}R(\Phi_3,\Phi_d)=-\frac{n^2}{3}.
\]  
\end{proof}

\section{Lie algebra derivations satisfying a polynomial identity}

Let $\Lg$ be a finite-dimensional Lie algebra over an arbitrary field $K$ of characteristic $p\ge 0$.
Denote by $c(\Lg)$ its nilpotency class and by $\Der(\Lg)$ its derivation algebra. 

\begin{defi}
Let $r\in K[t]$ be a polynomial. We say that a derivation $D\in \Der(\Lg)$ satisfies a {\em polynomial identity}
given by $r$ if $r(D)=0$. 
\end{defi}

An important example for such a polynomial identity is given by $r=t^n-1$. Then $D$ satisfies the polynomial
identity given by $r$ if and only if $D$ is a periodic derivation with $D^n=\id$.
Note that a periodic derivation is nonsingular. The existence of a nonsingular derivation already has a strong
implication on the structure of the Lie algebra. \\[0.2cm]
Let us first assume that $K$ has characteristic zero. Jacobson showed in  \cite{JAC} the following result.

\begin{prop}[Jacobson]\label{3.2}
Let $\Lg$ be a Lie algebra  over a field of characteristic zero admitting a nonsingular derivation. Then
$\Lg$ is nilpotent.  
\end{prop}  

In case the derivation is even periodic, one can hope to find in addition an effective bound for the nilpotency class.
Indeed, in Kostrikin and Kuznetsov \cite{KUK} showed the following result.

\begin{prop}
Let $\Lg$ be a Lie algebra  over a field of characteristic zero admitting a periodic derivation of order $n$
such that $n$ is not divisible by $6$. Then $\Lg$ is abelian.
\end{prop}

There is no conclusion here for the case that $6\mid n$. To fill the gap, we proved in \cite{BU45}
the following result over $\C$.

\begin{prop}\label{3.4}
Let $\Lg$ be a complex Lie algebra admitting a periodic derivation. Then $\Lg$ is nilpotent of class 
$c(\Lg)\le 2$.
\end{prop}  

In characteristic $p>0$ the situation is much more complicated. Jacobson's result is no longer true and even simple
modular Lie algebras may admit a periodic derivation for any prime characteristic. We summarize the following
classification result by Benkart, Kostrikin and Kuznetsov \cite{BKK,KOK}.

\begin{prop}
Let $\Lg$ be a simple modular Lie algebra over an algebraically closed field of characteristic $p>7$. Then the following
statements are equivalent.
\begin{itemize}
\item[$(1)$] $\Lg$ admits a periodic derivation.
\item[$(2)$] $\Lg$ admits a nonsingular derivation.
\item[$(3)$] $\Lg$ is either a special Lie algebra $S(m;{\bf n},\om_2)$, or a Hamiltonian Lie algebra 
$H(m;{\bf n},\om_2)$ as specified in \cite{BKK}.
\end{itemize}
\end{prop}

So we cannot conclude in general that a Lie algebras admitting a periodic derivation is nilpotent. On the other
hand we might be able to enforce nilpotency by adding further assumptions. Indeed, Kostrikin and Kuznetsov \cite{KUK}
showed that a modular Lie algebra of characteristic $p>0$ admitting a periodic derivation of order $n$ must be abelian
provided that $6\nmid n$ and that $p$ does not divide $\rho_n$. We generalize this result for Lie algebras in
arbitrary characteristic, including the case $6\mid n$.  

\begin{thm}\label{3.6}
Let $\Lg$ be a Lie algebra  over a field $K$ of arbitrary characteristic $p\ge 0$. Suppose that $\Lg$ admits a periodic
derivation $D$ of order $n$ such that $p$ does not divide $\rho_n$. Then $\Lg$ is nilpotent of class
\[
c(\Lg)\le \begin{cases} 1 & \text{ if } n\not\equiv 0\bmod 6, \\
                        2 & \text{ if } n\equiv 0\bmod 6,
\end{cases}  
\]
\end{thm}

\begin{proof}
We may assume that $K$ is algebraically closed since the nilpotency class is preserved under extensions of scalars.
Furthermore there exists a semisimple derivation $M$. For $p=0$ we may take $M=D$.
If $p>0$ then there exists an integer $k\ge 0$ such that $p^k\mid n$ and $\gcd(p,m)=1$ with
$m=\frac{n}{p^k}$. Let $M:=D^{p^k}$. Then $M$ is a periodic derivation of order $m$ dividing $n$. Since the order $m$ of
$M$ is coprime to $p$, the derivation $M$ is semisimple. Since $K$ is algebraically closed we can find an eigenbasis for
$\Lg$ with respect to $M$. Note that all eigenvalues $\la$ satisfy $\la^m=1$. \\[0.2cm]
{\em Case 1:} $6\nmid m$. Assume that $\Lg$ is not abelian. Then there exists eigenvectors $x,y\in \Lg$
with resepctive eigenvalues $\al,\be$ such that $[x,y]\neq 0$. It is easy to verify that $[x,y]$ is an eigenvector
with eigenvalue $\al+\be$. Hence we have $\al^m=\be^m=(\al+\be)^m=1$, so that the ratio $\frac{\al}{\be}$ is a common root
of the polynomials $t^m-1$ and $(t+1)^m-1$. By Lemma $\ref{2.5}$ we have $p\mid \rho_m$ and by Lemma $\ref{2.4}$
we obtain $p\mid\rho_m\mid \rho_n$ since $m\mid n$. This contradicts the assumption. Hence $\Lg$ is abelian. \\[0.2cm]
{\em Case 2:} $6\mid m$. Then $\rho_2=3$ and $\rho_3=2^2\cdot 7$ divide $\rho_m$ by Lemma  $\ref{2.5}$, so that $6\mid \rho_m$.
Hence $p>3$ by our assumptions. Assume that $[\Lg,\Lg],\Lg]\neq 0$. Then there exist eigenvectors $x,y,z$ with respective
eigenvalues $\al,\be,\ga$ such that $[[x,y],z]\neq 0$. We note that $[x,y]$ and $[[x,y],z]$ are also eigenvectors
with respective eigenvalues $\al+\be$ and $\al+\be+\ga$. By using the Jacobi identity we may further assume that $[z,x]$ is
an eigenvector with corresponding eigenvalues $\al+\ga$. But then the ratios $\frac{\al}{\be}$, $\frac{\al}{\ga}$ and
$\frac{\al+\be}{\ga}$ are all common roots of the polynomials $t^m-1$ and $(t+1)^m-1$. By  $\ref{2.5}$ these ratios
are roots of the polynomial $\Phi_3$, so that they have order $1$ or $3$. Since $p>3$, the order is always equal to $3$.
Let $\om\in K$ be an element of order $3$. Then there exists $1\le i,j,k\le 2$ such that $\be=\al\om^i$, $\ga=\al\om^j$
and $\al+\be=\ga\om^k$. By substitution we obtain that $1+\om^i-\om^{j+k}=0$, so that $\om$ is a common root
of $1+t+t^2$ and $1+t^i-t^{j+k}$. This implies that the resultant is zero in $K$, so that $p$ divides
$R(t^2+t+1,t^i-t^{j+k}+1)$, which is $1$ for all $i,j,k$ except for $(i,j,k)=(1,1,1),(2,2,2)$, where it is $4$.
It follows that $p=2$, which is a contradiction. Hence $c(\Lg)\le 2$. \\[0.2cm]
If $6\nmid n$ then also $6\nmid m$, so that we obtain the better bound $c(\Lg)\le 1$ by Case $1$.
\end{proof}  

\begin{rem}
For $p=0$ the assumption that $p\nmid \rho_n$ in the theorem is always satisfied since $\rho_n$ is nonzero.
Hence the result generalizes Proposition $\ref{3.4}$ from complex numbers to an arbitrary field of characteristic
zero. For $p>0$ there is no conclusion from the theorem for $p\mid \rho_n$. 
There exist both nilpotent and non-nilpotent Lie
algebras admitting a periodic derivation of order $n$ with $p\mid \rho_n$ for some $n$ and some $p$, see the
two examples below.
\end{rem}

\begin{ex}
Let $\Lg=W(1;m)$ be the Zassenhaus Lie algebra of dimension $2^m-1$ over $\F_2$ in characteristic $p=2$. It admits
a periodic derivation of order $n=2^m-1$, see \cite{BKK}. For all $m\ge 2$ we have $2\mid \rho_{2^m-1}$, so that
there is no conclusion from Theorem $\ref{3.6}$. And in fact $\Lg$ is simple and hence non-nilpotent.
\end{ex}

\begin{ex}
Let $\Lg$ be the free-nilpotent Lie algebra over $\F_3$ with $3$ generators of nilpotency class $2$.
It has a periodic derivation of order $6$. Since $3\mid \rho_6$ we cannot apply Theorem $\ref{3.6}$, but
nevertheless the conclusion holds. Indeed, $\Lg$ is $2$-step nilpotent.
\end{ex}

We will now generalize Theorem $\ref{3.6}$ from periodic derivations to derivations satisfying an arbitrary
polynomial identity. For this we use the methods and results from \cite{MON,MOE}. Recall
that a subset $X$ of an additive group $(G,+)$ is called {\em arithmetically-free} if $X$ does not contain
any arithmetic progression of the form $\al,\al+\be$,$\al+2\be,\ldots$, with $\al,\be\in X$. Let $H\colon \N\ra \N$
be the generalized Higman map, as defined in \cite{MON}.

\begin{thm}\label{3.10}
Let $K$ be a field of characteristic $p\ge 0$, $r\in K[t]$ be a polynomial of degree $n\ge 0$ and
$X=\{\al \in \ov{K} \mid r(a)=0 \}$ be the set of roots in $\ov{K}$.
If $X$ is an arithmetically-free subset of $(\ov{K},+)$, then every Lie algebra $\Lg$ over $K$ admitting
a derivation $D$, which satisfies $r(D)=0$, is nilpotent of class $c(\Lg)\le H(n)$. 
\end{thm}

\begin{proof}
Let $\Lh=\ov{K}\otimes_K\Lg$ and consider the derivation $M=\id\otimes D$ of $\Lh$. Then $r(M)=0$ in $\ov{K}$ and the
eigenspace decomposition of $\Lh$ with respect to $M$ is a grading by $(\ov{K},+)$, whose support is contained
in $X$. So we can apply Theorem $3.14$ of \cite{MON} to conclude that $\Lh$, and hence $\Lg$ is nilpotent of class
$c(\Lg)\le H(\abs{X})\le H(n)$.
\end{proof}

\begin{rem}\label{3.11}
If $X$ is {\em not} an arithmetically-free subset of $(\ov{K},+)$ in the above theorem, then there exists
some non-nilpotent Lie algebra over $\ov{K}$ of dimension $n+1$ admitting a derivation $D$, which satisfies $r(D)=0$. This follows
immediately from Proposition $3.8$ of \cite{MON}. Moreover we can construct then a filiform nilpotent Lie algebra of
arbitrarily high class admitting a derivation $D$ that satisfies $r(D)=0$, see Corollary $3.9$ in \cite{MON}.
\end{rem}

How can we decide for a given polynomial $r\in K[t]$ whether or not $X$ is arithmetically-free? \\[0.2cm]
For $p=0$ the answer is easy. $X$ is arithmetically-free if and only if $r(0)\neq 0$ in $K$.
So we obtain the following corollary.

\begin{cor}\label{3.12}
Let $K$ be a field of characteristic zero, $r\in K[t]$ be a polynomial of degree $n\ge 0$ such that
$r(0)\neq 0$. Then every Lie algebra $\Lg$ over $K$ admitting a derivation $D$, which satisfies $r(D)=0$,
is nilpotent of class $c(\Lg)\le H(n)$.  
\end{cor}

\begin{proof}
Assume that $r(0)\neq 0$ and let $\al,\be \in X$. Since the group  $(\ov{K},+)$ is torsion-free, the arithmetic
progression $\al, \al+1\cdot\be,1+2\cdot \be,\cdots$ contains infinitely many elements, so that it is not
contained in the finite set $X$. Hence $X$ is arithmetically-free. Conversely, let $X$ be arithmetically-free
and assume that $r(0)=0$. Then the arithmetic progression $0,0+1\cdot 0,0+2\cdot 0,\cdots$ is contained in $X$, so that
$X$ is not an arithmetically-free subset of $(\ov{K},+)$. This is a contraction, so that $r(0)\neq 0$.
\end{proof}  

\begin{rem}\label{3.13}
It is interesting to note that the corollary immediately implies Jacobson's result, Proposition $\ref{3.2}$. Indeed,
let $D$ be a nonsingular derivation. Then by Cayley-Hamilton the characteristic polynomial $r=\chi_D$ of $D$ satisfies
$r(D)=0$ with $r(0)\neq 0$ since $\det(D)\neq 0$. Hence the result follows.
\end{rem}

How can we decide for given $r$ whether or not $X$ is arithmetically-free in characteristic $p>0$?
This is much harder to answer for $p>0$ than for $p=0$. We will prove a result for $r=t^n-1\in \F_p[t]$.

\begin{defi}
 Let $r\in K[t]$ be a polynomial. The {\em period of $r$} is the minimal positive integer $m=\per(r)$ such that
 $r$ divides $t^m-1$ in $K[t]$, if such an $m$ exists. For $K=\F_p$ we define sets $\CP_p$ and $\CB_p$ by
 \begin{align*}
\CP_p & =\{\per(h (t^p-t))\mid h\in \F_p[t] \text{ with } h(0)\neq 0, \deg(h)\ge 1\},\\[0.2cm] 
\CB_p & =\N\cdot \CP_p.
\end{align*}
Note that a polynomial $r\in \F_p[t]$ has a period if and only if $r(0)\neq 0$.
\end{defi}  

\begin{ex}\label{3.15}
We have $p^k-1\in \CP_p$ for all primes $p$ and for all $k\ge 2$.
\end{ex}

To see this, let
\[
h=1+t^{p-1}+t^{p^2-1}+t^{p^3-1}+\cdots + t^{p^{k-1}-1}\in \F_p[t].
\]
Using
\[
(t^p-t)^{p^{\ell}-1} = (t(t^{p-1}-1))^{p^{\ell}-1} =t^{p^{\ell}-1}\cdot \frac{t^{p^{\ell}(p-1)}-1}{t^{p-1}-1}   
\]
we obtain
\[
 h(t^p-t) = 1+t^{p-1}+t^{2(p-1)}+t^{3(p-1)}+\cdots + t^{p^k-p} = \frac{t^{p^k-1}-1}{t^{p-1}-1}.
\]  
It is now easy to verify that $\per(h(t^p-t))=p^k-1$, which is contained in $\CP_p$. \\[0.2cm]
By computing the periods of $h(t^p-t)$ for a list of irreducible polynomials $h\in \F_p[t]$ of low degree
we obtain several elements in $\CP_p$.

\begin{ex}\label{3.16}
We have $3,7,31,73,85, 127 \in \CP_2$.
\end{ex}  

This follows from the table below, for $p=2$.
\vspace*{0.5cm}
\begin{center}
\begin{tabular}{c|c}
$h$ & $\per(h(t^2-t))$ \\
\hline
$t+1$ & $3$ \\
$t^3+t+1$ & $7$ \\
$t^4+t^3+t^2+t+1$ & $85$ \\
$t^5+t^2+1$ & $31$  \\
$t^7+t+1$ & $127$ \\
$t^9+t^4+t^2+t+1$ & $73$ \\
\end{tabular}
\end{center}
\vspace*{0.5cm}

We have the following result concerning the root set for $r=t^n-1\in \F_p[t]$. 
Note that this is implicitly stated in \cite{MA2}, section $4$. We state it here in an explicit form including
a proof for the convenience of the reader.

\begin{prop}\label{3.17}
Let $n$ be a positive integer and $p$ be a prime. Then the following assertions are equivalent.

\begin{itemize}
\item[$(1)$] The set $X_{n,p}=\{\al\in \ov{\F}_p \mid \al^n=1 \}$ is an arithmetically-free subset of $(\ov{\F}_p,+)$.
\vspace{0.1cm}  
\item[$(2)$] We have $h_{n,p}=\gcd(t^n-1,(t+1)^n-1,\ldots ,(t+p-1)^n-1)=1$ in $\F_p[t]$.
\vspace{0.1cm}
\item[$(3)$] We have $n\not \in \CB_p$.  
\end{itemize}  
\end{prop}

\begin{proof}
$(2)\Longrightarrow (1):$ Suppose that $X_{n,p}$ is not arithmetically-free. Then there exist $\al,\be\in X_{n,p}$
such that $\al^n=(\al+\be)^n=\cdots =(\al+(p-1)\be)^n=1$. For $\ga:=\frac{\al}{\be}$ we have $\ga^n=1$ and
$\ga\in X_{n,p}$. But then also $\ga^n=(\ga+1)^n=\cdots =(\al+(p-1))^n=1$, so that $\ga$ is a common root of
$t^n-1,(t+1)^n-1,\ldots ,(t+(p-1))^n-1$ and therefore of $h_{n,p}$. So $h_{n,p}\neq 1$, a contradiction. It follows
that $(1)$ holds. \\[0.2cm]
$(1)\Longrightarrow (2):$ Suppose that $h_{n,p}\neq 1$. Let $\ga$ be a root of it and let $\al=\be=1$. Then 
$\al,\al+\be,\ldots ,\al+(p-1)\be\in X_{n,p}$. Hence $X_{n,p}$ is not arithmetically-free. \\[0.2cm]
$(2)\Longrightarrow (3):$ Suppose that $n\in \CB_p$. Then we can choose an $h\in \F_p[t]$ with
$h(0)\neq 0$ and $\deg(h)\ge 1$ such that $m=\per (h(t^p-t))$ divides $n$. Hence $h(t^p-t)$ divides $t^n-1$.
Using Fermat's little theorem we see that $h(t^p-t)=h((t+\ell)^p-(t+\ell))$, so that $h(t^p-t)$ divides
$t^n-1,(t+1)^n-1,\ldots ,t+p-1)^n-1$ and therefore also $h_{n,p}$. it follows that $h_{n,p}\neq 1$. \\[0.2cm]
$(3)\Longrightarrow (2):$ Suppose that $h_{n,p}\neq 1$. Define polynomials $H_i\in \F_p[t]$ by $H_0=t^n-1$ and
\[
H_i=\gcd(H_{i-1}(t),H_{i-1}(t+1))
\]
for $i\ge 1$. Then for all $k\in \F_p$ we have $H_i(t+k)=\gcd(H_{i-1}(t+k),H_{i-1}(t+k+1))$. Hence we have for all
$i\ge 0$ that
\begin{align*}
h_{n,p} & = \gcd(H_i(t),H_i(t+1),\ldots,H_i(t+p-1)) \\
       & =  \gcd(H_{i+1}(t),H_{i+1}(t+1),\ldots,H_{i+1}(t+p-1)).
\end{align*}
Furthermore, $\deg(H_i)\ge \deg(H_{i+1})$ for all $i\ge 0$ so that there exists an $\ell\ge 1$ such that
$\deg(H_{\ell})=\deg(H_{\ell+1})$. Since $H_{\ell}(t)$ and $H_{\ell}(t+1)$ are monic polynomials of the same degree as their
greatest common divisor, we conclude that $H_{\ell}(t)=H_{\ell+1}(t)=H_{\ell}(t+1)$ and therefore
$H_{\ell}(t)=H_{\ell}(t+1)=\cdots =H_{\ell}(t+p-1)$, so that $h_{n,p}(t)=h_{n,p}(t+1)=\cdots =h_{n,p}(t+p-1)$.
Thus $h_{n,p}$ is of the form $h(t^p-t)$ for some $h\in \F_p[t]$. Since $h_{n,p}$ divides $t^n-1$, we have $h_{n,p}(0)\neq 0$
and therefore $h(0)\neq 0$. We have assumed that $h_{n,p}\neq 1$, so that $h$ is non-constant. So $\per(h(t^p-t))$ divides
$n$ and $n\in \CB_p$.
\end{proof}  

Theorem $\ref{3.10}$, Remark $3.11$ and Proposition $\ref{3.17}$ together yield the following result.
Note that if $\Lg$ admits a periodic derivation of order dividing $n$ then $\Lg\oplus K^n$ admits a periodic
derivation of order exactly $n$.

\begin{prop}\label{3.18}
Let $n\in \N$ and $p$ be a prime number. If $n\not\in \CB_p$, then every Lie algebra $\Lg$ over
a field of characteristic $p>0$ admitting a periodic derivation of order $n$ is nilpotent of class
$c(\Lg)\le H(n)$. If  $n\in \CB_p$ then there exists some non-nilpotent Lie algebra in characteristic
$p>0$ admitting a periodic derivation of order $n$.
\end{prop}  

\begin{rem}
The proposition shows that our set $\CB_p$ coincides with the set $\CN_p$ introduced by Shalev \cite{SHA} and
studied further by Mattarei \cite{MA1,MA2,MA3}. Shalev asked in his Problem $1$ in \cite{SHA} for the possible
orders $n$ of nonsingular derivations of non-nilpotent Lie algebras in characteristic $p>0$. Mattarei denoted
by $\CN_p$ the set of such positive integers $n$ and showed in \cite{MA2}, Theorem $2.1$ that $\CN_p$ can be described
in purely arithmetic terms, namely by 
\[
\CN_p=\{ n\in \N\mid \text{ there exists an element $\al\in \ov{\F}_p$ such that $(\al+\la)^n=1$ for all $\la\in \F_p$}\}.
\]  
Shalev already had shown that the $n$ arising as orders of periodic derivations of non-nilpotent Lie algebras in
characteristic $p>0$ belong to this set. Mattarei showed also the converse. \\[0.2cm]
Proposition $\ref{3.18}$ allows us to obtain examples of $n$ lying in $\CB_p$.
Any $h\in \F_p[t]$ with $h(0)\neq 0$ and $\deg(h)\ge 1$ will produce an element of $\CB_p=\CN_p$ by computing the
period of $h(t^p-t)$. For example, we can show that $3,7,31,73,85,127,\ldots, $ are elements
of $\CN_2$, see Example $\ref{3.16}$. These calculations have also been done in section $3.3$ of \cite{MA3}, where
an efficicient algorithm is used to test whether or not a given $n\in \N$ belongs to $\CB_p$.
\end{rem}

\begin{ex}
Fix a positive integer $n\le 12$. For such small $n$ we can decide for which primes $p$ we have $n\in \CB_p$.
In fact, $n\in \CB_p\Longrightarrow  p\mid \rho_n$ by Theorem $\ref{3.16}$, so that we only 
need to consider the prime
divisors $p$ of $\rho_n$ from the tables in section $2$. Using further results from \cite{SHA} we see that
\[
n\in \CB_p \Longleftrightarrow (n,p)\in \{ (3,2),(6,2),(7,2),(8,3),(9,2),(12,2)\}.
\]
\end{ex}
Hence we know, for example, that every modular Lie algebra of any prime characteristic $p>0$
admitting a periodic derivation of order $2,4,5,10,11$ is nilpotent. The ``bad'' primes for orders $n=3,6,7,8,9,12$
are $p=2,3$, where the Lie algebra over characteristic $p>0$ admitting a derivation of order $n$ need not be
nilpotent. \\[0.2cm]
For $p=2$ we can even describe the set $\CB_2$ totally in terms of $\rho_n$.

\begin{ex}
We have $n\in \CB_2 \Longleftrightarrow 2\mid \rho_n$. In fact, much more is true. The simple Lie algebra $W(1;2)$ of
dimension $3$ in characteristic $2$ admits a derivation $D$ with $D^n=\id$ for all $n$ with $2\mid \rho_n$.
\end{ex}

Let $(x_1,x_2,x_3)$ be a basis of $W(1;2)$ with $[x_1,x_2]=x_3$, $[x_1,x_3]=x_2$ and $[x_2,x_3]=x_1$.
Assume that $2\mid \rho_n$. Then there exists a $\la\in \ov{\F}_2$ such that $\la^n=(1+\la)^n=1$.
Define $D=\diag(1,\la,1+\la)$. Then $D$ is a derivation of $W(1;2)$ satisfying $D^n=\id$. Since
$W(1;2)$ is not nilpotent, it follows that $n\in \CB_2$. Conversely, $n\in \CB_2$ implies $2\mid \rho_n$
as above.

\begin{rem}
It would be also interesting to ask about the possible orders $n$ of nonsingular derivations of
{\em non-solvable} Lie algebras in characteristic $p>0$. The set of such positive integers $n$ would be contained
in the set $\CB_p$ and potentially be a proper subset.
\end{rem}

\section{Lie rings with a derivation satisfying a polynomial identity}

We can generalize Theorem $\ref{3.10}$ to Lie rings over $\Z$ by slightly modifying the structure theorems for Lie rings
proved in \cite{MON}. In addition to $\rho_n$ we also need the invariants $\delta(r)$ and $\sigma(r)$, introduced in
section $2$.

\begin{defi}
Let $L$ be a Lie ring over $\Z$ and $m \in \Z$. We say that $L$ has {\em no $m$-torsion} if the set of $v\in L$
such that there is a $k\in \N$ with $m^k\cdot v=0$ is equal to $\{0\}$. 
\end{defi}

First we obtain a direct analogue of Theorem $5.1.2$ in \cite{MON} as follows.

\begin{thm}\label{4.2}
Let $L$ be a Lie ring $L$ over $\Z$, $D$ be a derivation of $L$ and $r \in \Z[t]$ be a polynomial such that $r(D) = 0$.
Suppose that $L$ has no $\delta(r) \sigma(r)$-torsion. Then the Lie ring $\delta \cdot L$ can be embedded into a
Lie ring $M$ that is graded by $(\overline{\Q},+)$ and supported by the roots of the polynomial.
\end{thm}

\begin{proof}
The proof is identical to the proof of Theorem $5.1.2$ if we replace the invariant $\pi$ by $\sigma$ and pass to
the additive notion, i.e., replacing endomorphism by derivation, the semi-group $(\ov{\Q},\cdot)$ by $(\ov{\Q},+)$,
the product $\la\cdot \mu$ by the sum $\la+\mu$, and so on. We write $\de,\sigma$ for $\delta(r),\sigma(r)$.
The main idea of the proof goes as follows.
We first construct a larger Lie ring $\widetilde{L} := R \otimes_{\Z} L$ with coefficients in
$R := \Z[\lambda_1,\ldots,\lambda_k]$, where $\lambda_1,\ldots,\lambda_k$ are the roots of $r$ in $\overline{\Q}$.
This Lie ring admits a natural derivation $\map{\widetilde{D}}{\widetilde{L}}{\widetilde{L}}$
satisfying $r(\widetilde{D}) = 0$. We can then define generalized eigenspaces
$E_{\lambda_1} , \ldots , E_{\lambda_k}$ of $\widetilde{L}$ with respect to $\widetilde{D}$. They satisfy the obvious
grading property $[E_{\lambda_i},E_{\lambda_j}] \subseteq E_{\lambda_i + \lambda_j}$. We then define the $R$-module
$N := \sum_{1 \leq i \leq k} E_{\lambda_i}$ of $\widetilde{L}$. It is easy to verify that $N$ is a Lie ring over $R$.
Let $T$ be the $\delta \cdot \sigma$-torsion ideal of $\widetilde{L}$. We then consider the natural
projection $\map{P}{\widetilde{L}}{\widetilde{L}/T}$. Each of the eigenspaces $E_{\lambda_i}$ contains $T$.
By defining $M_{\lambda_i} := P(M_{\lambda_i})$, we obtain
\[
M := P(N) = \sum_{1 \leq i \leq k} P({N}_{\lambda_i}) = \sum_{1 \leq i \leq k} M_{\lambda_i}.
\]
We can then verify that this sum is direct. So
\[
M = \bigoplus_{1 \leq i \leq k} M_{\lambda_i}
\]  
is in fact a grading of $M$ by $(\overline{\Q},+)$. Since $L$ is assumed to have no
$\delta \cdot \sigma $-torsion, the composition of $\iota \colon \de\cdot L\ra N \subseteq \widetilde{L}$, given by
$\delta \cdot v \mapsto \delta \otimes v$, and $\map{P}{\widetilde{L}}{M}$ is the required
embedding of Lie rings over $\Z$.
\end{proof}

Now we can formulate an analogue of Theorem $1.2.6$ of \cite{MON} as follows.

\begin{thm}
Let $L$ be a Lie ring over $\Z$, $D$ be a derivation of $L$ and $r\in \Z[t]$ be a polynomial of degree $n$
such that $r(D)=0$. If $r(0)\neq 0$ and if $L$ has no $\delta(r)\sigma(r)$-torsion, then $L$ is
nilpotent of class $c(L)\le H(n)$.
\end{thm}

\begin{proof}
Since $L$ has no $\de$-torsion we obtain that $c(L)=c(\de\cdot L)$. By Theorem $\ref{4.2}$ we can embed the Lie ring
$\de\cdot L$ into a Lie ring $M$ that is graded by $(\ov{\Q},+)$
and supported by the root set $X$ of $r$. Since $r(0)\neq 0$ this support $X$ is arithmetically-free. According to  
Theorem $3.14$ of \cite{MOE} $M$ is nilpotent of class $c(M)\le H(n)$. Hence $\de\cdot L$ is nilpotent of class
at most $H(n)$, too. 
\end{proof}

\section*{Acknowledgments}
Dietrich Burde is supported by the Austrian Science Foun\-da\-tion FWF, grant I3248 and P33811. 
Wolfgang A. Moens acknowledges support by the Austrian Science Foun\-da\-tion FWF, grant P30842.

\end{document}